\documentclass{amsart}

\usepackage{amssymb, amsthm, amsfonts, amsmath}
\usepackage{hyperref}
\usepackage{enumitem}

\newtheorem{thm}{Theorem}[]
\newtheorem{proposition}[thm]{Proposition}
\newtheorem{lemma}[thm]{Lemma}

\theoremstyle{definition}

\theoremstyle{remark}
\newtheorem{remark}[thm]{Remark}

\newcommand{\R}{\mathbb{R}}
\newcommand{\C}{\mathbb{C}}
\newcommand{\CP}{\mathbb{CP}}
\newcommand{\mF}{\mathcal{F}}
\newcommand{\mO}{\mathcal{O}}
\newcommand{\p}{\partial}

\DeclareMathOperator{\Hom}{Hom}
\DeclareMathOperator{\End}{End}
\DeclareMathOperator{\Aut}{Aut}

\DeclareMathOperator{\pardeg}{par-deg}

\begin{document}

\title{Parabolic bundles and spherical metrics}

\author{Martin de Borbon}
\address{King's College London, Department of Mathematics\\
	Strand, London, WC2R 2LS, United Kingdom}
\email{martin.deborbon@kcl.ac.uk, dmitri.panov@kcl.ac.uk}

\author{Dmitri Panov}

\begin{abstract}
We use the Kobayashi-Hitchin correspondence for parabolic bundles to reprove the results of Troyanov \cite{Troy} and Luo-Tian \cite{LT} regarding existence and uniqueness of conformal spherical metrics on the Riemann sphere with prescribed cone angles in the interval \((0, 2\pi)\) at a given configuration of three or more points. 
\end{abstract}

\maketitle

\section{Introduction}

\subsection{Main result}
Fix \(n \geq 3\) distinct points \(x_i\) in \(\CP^1\)  together with real numbers \(0<\alpha_i<1\). Assume that the following numerical conditions are satisfied. 
\begin{itemize}
	\item Gauss-Bonnet:
	\begin{equation}\label{eq:GB}
	\sum_i (1-\alpha_i) < 2 .
	\end{equation}
	\item Stability: for every \(1 \leq j \leq n\) we have
	\begin{equation}\label{eq:S}
	1-\alpha_j < \sum_{i \neq j} (1-\alpha_i).
	\end{equation}
\end{itemize}

We apply the theory of parabolic bundles to prove the next.

\begin{thm}[{\cite{Troy, LT}}]\label{thm}
	There is a unique conformal spherical metric \(g\) on \(\CP^1\) with cone angles \(2\pi\alpha_i\) at \(x_i\).
\end{thm}

The statement of Theorem \ref{thm} has the following meaning.
For every point in the open set \(\CP^1 \setminus \{x_1, \ldots, x_n\}\) we can find a centred complex coordinate \(z\) such that \(g\) is equal to the constant curvature \(1\) metric on the Riemann sphere
\begin{equation}\label{eq:sphereformula}
	\frac{2}{1+|z|^2} |dz| .
\end{equation} 
On the other hand,
at a cone point \(x_i\) we can find a centred complex coordinate \(z\) such that \(g\) is the pull-back of \eqref{eq:sphereformula} by the map \(z \mapsto z^{\alpha_i}\).

\begin{remark}
	The existence part of Theorem \ref{thm} was first proved by Troyanov \cite{Troy} via variational PDE methods. Soon after Luo-Tian \cite{LT} gave a more geometric existence proof appealing to Alexandrov's embedding theorem, moreover they also established uniqueness. The story for spherical metrics with cone angles bigger than \(2\pi\) is far more intricate, see \cite{eremenko} for an up to date survey on the subject.
\end{remark}

\subsection{Related work}
Lingguang Li, Jijian Song and Bin Xu \cite{LSX2} study spherical metrics on Riemann surfaces of genus \(\geq 1\) with cone angles equal to integer multiples of \(2\pi\) using  rank \(2\) holomorphic vector bundles. Semin Kim and Graeme Wilkin \cite{KW} analyse hyperbolic metrics on Riemann surfaces of genus \(\geq 2\) with cone angles in the interval \((0, 2\pi)\) by means of parabolic Higgs bundles.

\begin{remark}
	Other proofs of Theorem \ref{thm} can be derived from more recent developments on singular K\"ahler-Einstein metrics, see the survey \cite{Rubinstein} and references therein. (A new proof of uniqueness is also given in \cite{Bartoluccietal}.)
	Our main motivation is to extend the parabolic bundle technique to higher dimensions and apply it in the study of constant holomorphic sectional curvature metrics with cone singularities at complex hypersurfaces, as started in \cite{Pan} for the flat case.
\end{remark}

\subsection{Outline}
Sections \ref{sect:bundle}, \ref{sect:flag}, \ref{sect:parab} and \ref{sect:stab} construct a stable rank \(2\) parabolic bundle \(E_{*}\) with weights \(a_{i1}, a_{i2}\) at \(x_i\) determined by  \(\alpha_i\). In Section \ref{sect:conect} we apply the Kobayashi-Hitchin correspondence to obtain a  unitary logarithmic connection \(\nabla\) adapted to the parabolic structure. In Section \ref{sect:fol} we analyse the singular foliation \(\mF\) on the projectivized bundle \(\mathbb{P}(E)\) defined by the horizontal distribution of the connection \(\nabla\). In Section \ref{sect:sec} we show that there is a unique section \(\sigma: \CP^1 \to \mathbb{P}(E)\) transversal to \(\mF\), its image is the unique curve with negative self-intersection of the Hirzebruch surface \(\mathbb{P}(E)\). In Section \ref{sect:sphmet} we obtain \(g\) as the pull-back by \(\sigma\) of the Fubini-Study metric on the fibres (normalized to have constant curvature \(1\)).

Uniqueness is established in Section \ref{sect:uniq} by reversing the arguments used in the existence part.
A conformal spherical metric \(g'\) with cone angles \(2\pi\alpha_i\) at the points \(x_i\) induces a holomorphic unitary connection \(\nabla'\) on the trivial rank \(2\) bundle over the punctured sphere. We extend \(\nabla'\) over the points \(x_i\) with logarithmic singularities and semi-simple residues with eigenvalues \(a_{i1}, a_{i2}\). We obtain a semi-stable parabolic vector bundle \(E'_{*}\). We show that there is a (unique up to scale) isomorphism between \(E'_*\) and \(E_*\). The developing map of \(g'\) defines a section \(\sigma': \CP^1 \to \mathbb{P}(E')\) which is identified with \(\sigma\) under the above isomorphism. Since \(E_*\) is stable, the uniqueness part of the Kobayashi-Hitchin correspondence guarantees that the connections \(\nabla'\) and \(\nabla\) agree under the isomorphism hence \(g'=g\).

\subsection*{Acknowledgments}
DP thanks Gabriele Mondello for many years of discussions on spherical metrics.
This work was supported by 
EPSRC Project EP/S035788/1, \emph{K\"ahler manifolds of constant curvature with conical singularities}.

\section{Main constructions}
We access the spherical metric by using a well known fibre bundle description of projective structures, which in general lines can be described as follows.
A geometric structure on a manifold \(M\) modelled on \(A\) with transition functions in \(\mathcal{A}\) can be viewed as a fibre bundle over \(M\) with fibre \(A\) and structure group \(\mathcal{A}\), together with a section and a foliation transverse to the fibres and to the section, see \cite{SullThu}. For projective structures on Riemann surfaces we have \(A=\CP^1\) with \(\mathcal{A} = PSL(2, \C)\) and all objects involved are holomorphic.
With this background in mind, we proceed as explained in the outline. 

\subsection{The bundle}\label{sect:bundle}
Consider the rank \(2\) vector bundle
\begin{equation}\label{eq:vbE}
E = \mO(1) \oplus \mO(n-1) .
\end{equation}

Write \(\Aut(E)\) for the automorphism group of \(E\).
It is the subset of \(H^0(\End E)\) made of elements \(\Phi\) of the form
\begin{equation}\label{eq:autE}
\Phi =
\begin{pmatrix}
\lambda_1 & 0 \\
P & \lambda_2
\end{pmatrix}
\end{equation}
with \(\lambda_1, \lambda_2 \in \C^*\) and \(P \in H^0(\mO(n-2))\).

\begin{lemma}\label{lem:subbundles}
	Let \(L \subset E\) be a line bundle of positive degree \(d = \deg L\).
	\begin{enumerate}[label=(\alph*)]
		\item If \(d >1\) then
		\(L\) is equal to the direct summand \(\mO(n-1)\).
		\item If \(d =1\) then \(L\) is the image of  \(\mO(1)\) by an element of \(\Aut(E)\).
	\end{enumerate}
\end{lemma}

\begin{proof}
	Consider the two projections of \(L\) to the components \(\mO(1)\) and \(\mO(n-1)\) of \(E\) given by Equation \eqref{eq:vbE}. We obtain two sections
	\[\sigma_1 \in H^0(L^{*} \otimes \mO(1)) \hspace{2mm} \mbox{ and } \hspace{2mm} \sigma_2 \in H^0(L^* \otimes \mO(n-1)) \]
	with no common zero. 
	
	If \(d>1\) then \(L^* \otimes \mO(1)\) has negative degree, so \(\sigma_1 = 0\) and \(L=\mO(n-1)\).
	
	If \(d=1\) then \(\sigma_1\) is an isomorphism between \(L\) and \(\mO(1)\). 
	The automorphism \(\Phi\) given by Equation \eqref{eq:autE} with \(\lambda_1=\lambda_2=1\) and \(P = \sigma_2 \circ \sigma_1^{-1}\)
	takes \(\mO(1)\) to \(L\).
\end{proof}

\subsection{The flag}\label{sect:flag}
We consider tuples \( \mathbf{F}  =\{F_1, \ldots, F_n\}\) of lines \(F_i \subset E_{x_i}\) (where \(E_x\) denotes the fibre of \(E\) at the point \(x\)) up to the action of \(\Aut(E)/ \C^*\).
We show that there is a unique open orbit on which the action is free and transitive. 

\begin{lemma}\label{lem:flag}
	Up to the action of \(\Aut(E)\) there is a unique set \(\mathbf{F}  =\{F_1, \ldots, F_n\}\) of lines \(F_i \subset E_{x_i}\) satisfying the following properties:
	\begin{itemize}
		\item[(i)] the intersection of \(F_i\) with \(\mO(n-1)\) is zero for all \(i\);
		\item[(ii)] there is no degree \(1\) line bundle \(L \subset E\) that contains all \(F_i\).
	\end{itemize}
\end{lemma}

\begin{proof}
	Define \(\mathbf{F}=\{F_1, \ldots, F_n\}\) where
	\(F_i = \) fibre of \(\mO(1)\) at \(x_i\) for \(i=1, \ldots, n-1\) and \(F_n =  \C \cdot v_n\) where \(v_n\) is a fixed vector in
	\(E_{x_n}\) which does not belong to neither \(\mO(1)\) nor \(\mO(n-1)\). It is clear that \(\mathbf{F}\)  satisfies \((i)\).
	
	Let \(L \subset E\) be a degree one line bundle with \(F_i \subset L\) for \(i=1, \ldots, n-1\). By Lemma \ref{lem:subbundles} we have \(L = \Phi \cdot \mO(1)\) with \(\Phi \in \Aut(E)\) of the form given by Equation \eqref{eq:autE}. 
	Since \(L\) agrees with \(\mO(1)\) at \(n-1\) points it follows that \(P=0\). Hence \(L=\mO(1)\) and therefore \(F_n \not\subset L\). We conclude that \(\mathbf{F}\) satisfies \((ii)\).

	Suppose \(\mathbf{F}'\) is a collection of lines satisfying properties \((i)\) and \((ii)\). We claim that there is a (unique up to dilation) automorphism \(\Phi\) that sends \(\mathbf{F}'\) to \(\mathbf{F}\). Take non-zero vectors \(v'_i\) such that \(F'_i = \C \cdot v_i'\). Let \(\sigma_1 v'_i\) and \(\sigma_2 v'_i\) be the respective projections to \(\mO(1)\) and \(\mO(n-1)\). Property \((i)\) implies that \(\sigma_1 v_i' \neq 0\) for all \(i\).
	The evaluation map that takes a section of \(\Hom(\mO(1), \mO(n-1))\) to its values at the first \(n-1\) projections
	\[P \mapsto (P(\sigma_1 v'_1), \ldots, P(\sigma_1 v'_{n-1}) )\]
	defines a linear isomorphism 
	\[H^0(\mO(1)^* \otimes \mO(n-1)) \xrightarrow{\sim} \mO(n-1)_{x_1} \times \ldots \times \mO(n-1)_{x_{n-1}} . \]
	Let \(P \in H^0(\mO(n-2))\) be (the unique section) such that \(P(\sigma_1 v_i') = -\sigma_2 v'_i \) for \(i=1, \ldots, n-1\). The endomorphism
	\[ \Phi =  
	\begin{pmatrix}
	1 & 0 \\
	P & 1
	\end{pmatrix}
	\in \Aut(E)
	\]
	sends the first \(n-1\) lines \(F_i'\) to \(\mO(1)\). Property \((ii)\) implies that \(\Phi F'_n \not\subset \mO(1)\), i.e. \(\sigma_2 \Phi v_n' \neq 0\).
	Since \(\sigma_1 \Phi v_n' = \sigma_1 v_n' \neq 0 \),
	we can take  \(\lambda_1, \lambda_2 \in \C^*\) so that the composition
	\[
	\begin{pmatrix}
	\lambda_1 & 0 \\
	0 & \lambda_2
	\end{pmatrix}
	\circ \Phi
	\]
	maps \(\mathbf{F}'\) to \(\mathbf{F}\).
\end{proof}

\subsection{The parabolic structure}\label{sect:parab}
We follow the conventions adopted in Mochizuki's book \cite[Chapter 3]{Mochizuki}. See also \cite[Section 6]{Pan}.\footnote{Section 6 in \cite{Pan} considers the case where the base manifold has complex dimension \(2\). The same results hold for our complex dimension \(1\) case.} 

For each \(\alpha_i\) we define two weights given by
\begin{equation}\label{eq:weights}
a_{i1} = \frac{1-\alpha_i}{2} \hspace{2mm} \mbox{ and } \hspace{2mm} a_{i2} = \frac{1+\alpha_i}{2} .
\end{equation}
The weights \(0 < a_{i1} < a_{i2} < 1\) satisfy \(a_{i1}+a_{i2} =1\) and \(a_{i2}-a_{i1}=\alpha_i\).

The flag \(\mathbf{F}=\{F_1, \ldots, F_n\}\) of Lemma \ref{lem:flag} together with the weights given by Equation \eqref{eq:weights} determine a parabolic bundle
\(E_*\). Following \cite[Definition 6.1]{Pan} the parabolic  structure is made of increasing filtrations of \(E\) by locally free sub-sheaves \(E^i_t\) for \(0 < t \leq 1\). The sheaves \(E^i_t\) are defined as:
\begin{itemize}
	\item \(0 < t < a_{i1}\) sections of \(E\) that vanish at \(x_i\);
	\item \(a_{i1} \leq t < a_{i2}\) sections of \(E\) tangent to \(F_i\);
	\item \(a_{i2} \leq t \leq 1\) sections of \(E\).
\end{itemize}
The families \(E^i_t\) can be extended to the whole \(\R\) by the rule \(E^i_{t+1}=E^i_t\otimes \mO(x_i)\).

\begin{remark}
	In our conventions
	the smallest weight \(a_{i1}\) at \(x_i\) is associated with the line \(F_i\) and the largest weight \(a_{i2}\) is associated with \(E_{x_i}/F_i\), see \cite[Section 2]{LSS}. Note that the articles \cite{MS, Biquard} use a convention different from ours. 
\end{remark}

\begin{lemma}
	The parabolic degree of \(E_{*}\) is zero.
\end{lemma}

\begin{proof}
	According to \cite[Definition 6.3]{Pan} we have
	\begin{align*}
	\pardeg E_{*} &= \deg E - \sum_i (a_{i1}+a_{i2}) \\
	&=  n - \sum_{i=1}^{n}1 = 0 . \qedhere
	\end{align*}
\end{proof}

\subsection{Stability}\label{sect:stab}
Let \(L \subset E\) be a line bundle. The parabolic structure on \(E\) induces one on \(L\) and we have
\begin{equation}\label{eq:pardegL}
\pardeg L_{*} = \deg L - \sum_{F_i \subset L} a_{i1} - \sum_{F_i \not\subset L} a_{i2} .
\end{equation}
The more lines \(F_i \subset L\) the bigger \(\pardeg L_{*}\) is.

\begin{proposition}\label{prop:stability}
	The parabolic bundle \(E_{*}\) is stable, i.e. for every line bundle \(L \subset E\) we have \(\pardeg L_{*}<0\).
\end{proposition}

\begin{proof}
	Because \(\pardeg L_{*} < \deg L\),
	it is enough to consider \(L \subset E\) with positive degree, as described in Lemma \ref{lem:subbundles}. 
	
	Suppose that \(L=\mO(n-1)\). By our choice of flag, since \(F_i \not\subset L\) for all \(i\), Equation \eqref{eq:pardegL} gives us
	\begin{align*}
	\pardeg L_{*} &= n-1 - \frac{1}{2} \sum_i (1+\alpha_i) \\
	&=
	-1+\frac{1}{2} \sum_i (1- \alpha_i) .
	\end{align*}
	We see that \(\pardeg L_{*}<0\) is equivalent to the Gauss-Bonnet condition \eqref{eq:GB}.
	
	Suppose that \(\deg L=1\). By our choice of flag there are at most \(n-1\) lines \(F_i \subset L\). Let \(1 \leq j \leq n\) be such that \(F_j \not\subset L\). Using Equation \eqref{eq:pardegL} we obtain
	\begin{align*}
	\pardeg L_{*} &\leq 1 - \frac{1+\alpha_j}{2} - \frac{1}{2} \sum_{i \neq j} (1-\alpha_i) \\
	&= \frac{1-\alpha_j}{2} - \frac{1}{2} \sum_{i \neq j} (1-\alpha_i) 
	\end{align*}
	and \(\pardeg L_{*}<0\) follows from the stability condition \eqref{eq:S}.
\end{proof}

\subsection{The connection}\label{sect:conect}
Kobayashi-Hitchin correspondence gives us the next.

\begin{proposition}
	There is a meromorphic unitary connection \(\nabla\) on \(E\) with simple poles at the points \(x_i\) and holomorphic otherwise. Close to \(x_i\) there is a holomorphic trivialization of \(E\) such that
	\begin{equation}\label{eq:modelconection}
	\nabla = d -  
	\begin{pmatrix}
	a_{i1} & 0 \\
	0 & a_{i2}
	\end{pmatrix} \frac{dt}{t} 
	\end{equation}
	and \(F_i \subset E_{x_i}\) is equal to the \(a_{i1}\)-eigenspace of the residue endomorphism.
\end{proposition}

\begin{proof}
	Since \(E_{*}\) is stable and has zero parabolic degree,
	the Mehta-Seshadri theorem (see \cite[Theorem 4.1]{MS} and \cite[Th\'eor\`eme 2.5]{Biquard}\footnote{See also \cite[Section 3]{Simpson} for the relation with logarithmic connections in the more general case of parabolic Higgs bundles.}) implies that there is a unique flat unitary irreducible connection \(\nabla\) compatible with the parabolic structure. We recall the properties of this connection, along the lines of \cite[Theorem 6.9]{Pan} specialized to the Riemann surface case.
	
	The compatibility with the parabolic structure (see \cite[Definition 6.8]{Pan}) implies that \(\nabla\) extends over the punctures with logarithmic singularities. More precisely, in a (hence every) local trivialization of \(E\) around \(x_i\) we have
	\[\nabla = d - A(t) \frac{dt}{t},\]
	where \(t\) is a (hence any) complex coordinate centred at \(x_i\) and \(A(t)\) is a holomorphic matrix valued function. 
	Moreover, the residue \(A(0) \in \End(E_{x_i})\) (which is independent from trivialization and local coordinate choices) has eigenvalues \(a_{i1}, a_{i2}\) and \(F_i\) is equal to the \(a_{i1}\)-eigenspace. 
	
	Since the difference of eigenvalues \(a_{i2}-a_{i1} = \alpha_i\) is non-integer, Equation \eqref{eq:modelconection} follows from the normal form theorem for non-resonant Fuchsian singularities \cite[Theorem 16.16]{IY}.
\end{proof}


\subsection{The foliation}\label{sect:fol}
Consider the projectivization \(\mathbb{P}(E)\) of the vector bundle \(E\) given by Equation \eqref{eq:vbE}. We obtain the Hirzebruch surface 
\begin{equation}
	\Sigma_{n-2} = \mathbb{P} \left( \mO \oplus \mO(n-2)\right) .
\end{equation}
Write \(\Pi: \Sigma_{n-2} \to \CP^1\) for the projection map with \(\CP^1\)-fibres.

The horizontal distribution of \(\nabla\) defines a holomorphic foliation \(\mF\) on the restriction of \(\mathbb{P}(E)\) over the open set
\[U = \CP^1 \setminus \{x_1, \ldots, x_n\} . \]
Let \(v \in \mathbb{P}(E)\) with
\(\Pi(v) \in U\),
take \(\tilde{v} \in E\) on the line represented by
\(v\) and let \(s\) be a flat local section of \(E\)
that takes the value \(\tilde{v}\) at \(\Pi(v)\). By definition, the leaf of \(\mF\) at \(v\) is locally given by the projection to \(\mathbb{P}(E)\) of the graph of \(s\). The foliation \(\mF\) is transversal to the \(\CP^1\)-fibres of \(\Pi\).

\begin{lemma}\label{lem:extF}
	\(\mF\) extends to a singular foliation on \(\Sigma_{n-2}\) that is tangent to fibres over the points \(x_i\). At each fibre \(\Pi^{-1}(x_i)\) the foliation \(\mF\) has two singularities \(v_{i1}\) (\(= F_i\)) and \(v_{i2}\).
\end{lemma}

\begin{proof}
	Let \(t\) be a complex coordinate on the base centred at \(x_i\). Take a trivialization of \(E\) close to \(x_i\) defined by a holomorphic frame \(s_1, s_2\) such that \(\nabla\) is equal to the local model connection given by Equation \eqref{eq:modelconection}. 
	Set 
	\[v_{i1} = \C \cdot s_1(x_i) \mbox{ and } v_{i2} = \C \cdot s_2(x_i) .\]
	Let \((y_1, y_2)\) be the  fibre coordinates
	\( (y_1, y_2) \mapsto y_1 s_1 + y_2 s_2 .\)
	The flat sections are of the form 
	\[t \mapsto \begin{pmatrix}
	y_1 = c_1 t^{a_{i1}} \\
	y_2 = c_2 t^{a_{i2}}
	\end{pmatrix} \mbox{ with } c_1, c_2 \in \C .\]
	
	On the open set \(\mathbb{P}(\{y_1 \neq 0\}) \subset \mathbb{P}(E)\) we have local coordinates
	\((t, y = y_2/y_1)\) centred at \(v_{i1}\). In these coordinates the leaves of the foliation \(\mF\) are given by graphs \(y=ct^{\alpha_i}\) (here \(c \in \C\) and \(\alpha_i=a_{i2}-a_{i1}\)) corresponding to flat local sections of the trivial bundle with connection \(d - (\alpha_i/t) dt\).
	Equivalently, \(\mF\) is made of the orbits of the linear vector field \(t \p_t + \alpha_i y \p_y\). Similarly, on \(\mathbb{P}(\{y_2\neq 0\})\) we have coordinates \((t, y=y_1/y_2)\) centred at \(v_{i2}\) such that the leaves of \(\mF\) are orbits of the vector field \(t \p_t - \alpha_i y \p_y\). In other words, close to \(v_{i2}\) the foliation \(\mF\) is made of integral submanifolds of the horizontal distribution given by the connection \(d + (\alpha_i/t)dt\). 
\end{proof}

\begin{remark}
	Our (singular) foliation \(\mF\) on \(\Sigma_{n-2}\) is of Riccati type, meaning that it is transversal to the generic fibres of \(\Sigma_{n-2} \to \CP^1\). See \cite{LorayPerez} for the relation between Riccati foliations and projective structures on Riemann surfaces.
\end{remark}

\subsection{The section}\label{sect:sec}
The line bundle \(\mO(n-1) \subset E\) defines a section
\(\sigma\) of \(\mathbb{P}(E)\) with self-intersection
\begin{equation}\label{eq:selfint}
	\sigma^2 = - (n-2) .
\end{equation}
As a matter of fact, \(\sigma\) is the unique section of the Hirzebruch surface \(\mathbb{P}(E) =\Sigma_{n-2} \to \CP^1 \) with negative self-intersection.

\begin{lemma}\label{lem:transverse}
	The section \(\sigma\) is everywhere transversal to the foliation \(\mF\). In particular, \(\sigma\) does not go through any of the \(2n\) singularities of \(\mF\).
\end{lemma}

\begin{proof}
	Let \(T_{\mF}\) be the tangent bundle of \(\mF\), see \cite[Chapter 2]{Brunella}. Given a curve \(C \subset \Sigma_{n-2}\)
	the total number of tangencies of \(C\) with \(\mF\) is given by
	\begin{equation}\label{eq:tangencies}
		\mbox{Tan}(\mF, C) = C^2 - C \cdot T_{\mF} , 
	\end{equation}
	see \cite[Proposition 2.2]{Brunella}. 
	Lifting a holomorphic vector field \(V\) on \(\CP^1\) we obtain a section of \(T_{\mF}\) with zeros at the fibres
	over \(\{V=0\}\) and simple poles at the fibres over \(x_i\). It follows that, up to linear equivalence,
	\begin{equation}\label{eq:tF}
		T_{\mF} \cong -(n-2) \cdot \mathfrak{f}
	\end{equation}
	where \(\mathfrak{f}\) represents the class of a fibre
	(see \cite[p. 48]{Brunella}).
	Since \(\sigma\) is a section, we have \(\sigma \cdot \mathfrak{f} =1\). Equations \eqref{eq:tangencies} and \eqref{eq:tF} imply that
	\begin{equation}\label{eq:tansigma}
		\mbox{Tan}(\mF, \sigma) = \sigma^2 + (n-2) .
	\end{equation}
	Equations \eqref{eq:selfint} and \eqref{eq:tansigma} give us \(\mbox{Tan}(\mF, \sigma) = 0\). The statement of the lemma follows from the vanishing of the total number of tangencies.
\end{proof}

\section{Proof of Theorem \ref{thm}: The spherical metric}

\subsection{Existence}\label{sect:sphmet}
We use the section \(\sigma\) to pull-back the Fubini-Study metric on the fibres, as we explain next.

Let \(e_1, e_2\) be a local holomorphic unitary frame of \(E\) defined on an open subset 
\[V \subset U = \CP^1 \setminus \{x_1, \ldots, x_n\} .\] 
In this trivialization, the section \(\sigma\) is defined by a holomorphic function \(f: V \to \CP^1\). The fact that \(\sigma\) is transverse to \(\mF\) implies that \(f\) has nowhere zero derivative, hence we can pull-back the round metric on \(\CP^1\) by \(f\) to obtain a smooth conformal metric on \(V\) of constant curvature \(1\). The invariance of the Fubini-Study metric under the action of \(U(2)\) by linear unitary transformations implies that its pull-back is independent of the trivializing frame. Hence, we obtain a smooth conformal spherical metric \(g\) on \(U\).

\begin{lemma}\label{lem:coneangle}
	The metric \(g\) has a cone singularities of total angle \(2\pi\alpha_i\) at the points \(x_i\). Hence, the existence part of Theorem \ref{thm} follows.
\end{lemma}

\begin{proof}
	Let \((t, y)\) be coordinates on \(\mathbb{P}(E)\) centred at \(v_{i2}\) as in the proof of Lemma \ref{lem:extF}. The leaves of \(\mF\) are graphs \(y=ct^{-\alpha_i}\) of flat local sections of the trivial rank one bundle over \(\C\) with connection
	\[d + \frac{\alpha_i}{t} dt . \]
	Lemma \ref{lem:transverse} implies that
	the section \(\sigma\) is represented by a holomorphic map \(y=f(t)\) with \(f(0)\neq 0\).
	By construction, the metric \(g\) close to \(x_i\) is obtained by pulling-back the Fubini-Study metric on the fibre at \(\{t=1\}\), say, under the composition
	\[(t,0) \mapsto (t, f(t)) \mapsto \mbox{ fibre at } \{t=1\}  \]
	where the second arrow is the (multivalued) map given by transport along the leaves of \(\mF\).
	The composition results into a (multivalued) map from a neighbourhood of the base \(t=0\) to a neighbourhood of \(y=0\) over the fibre at \(\{t=1\}\) given by
	\begin{equation}\label{eq:mapfibre}
		t \mapsto t^{\alpha_i} f(t) .
	\end{equation}
	Since \(f(0) \neq 0\) we can take a local holomorphic branch of \(f^{1/\alpha_i}\) close to the origin. Introducing the local coordinate \(z= t f(t)^{1/\alpha_i}\) centred at \(x_i\), we
	deduce that \(g\) is the pull-back of the Fubini-Study metric by the map \(z \mapsto z^{\alpha_i}\).
\end{proof}

\subsection{Uniqueness}\label{sect:uniq}
Let \(g'\) be a conformal spherical metric on \(\CP^1\) with cone angles \(2\pi\alpha_i\) at the points \(x_i\). 
We will show that \(g'\) is equal to the metric \(g\) constructed in the previous Section \ref{sect:sphmet}.

\begin{lemma}
	The metric \(g'\) defines a holomorphic unitary connection \(\nabla'\) on a (trivial) rank \(2\) bundle \(E'_U\) over \(U=\CP^1 \setminus \{x_1, \ldots, x_n\}\) with holonomy at a positive loop \(\gamma_i\) around \(x_i\) conjugate to
	\begin{equation}\label{eq:holloop}
	\begin{pmatrix}
	\exp(\pi i (1-\alpha_i)) & 0 \\
	0 & \exp(\pi i (1+\alpha_i))
	\end{pmatrix} .
	\end{equation}
\end{lemma}

\begin{proof}
	By \cite[Proposition 2.12]{MP}, the monodromy of \(g'|_U\) into \(SO(3)\) has a canonical lift \(\rho: \pi_1(U) \to SU(2)\) with \(\rho(\gamma_i)\) conjugate to \eqref{eq:holloop}. 
	This lift induces a holomorphic connection \(\nabla'\) on the rank two vector bundle 
	\[E'_U = \tilde{U} \times_{\pi_1(U)} \C^2\] 
	over \(U\), where 
	\(\pi_1(U)\) acts on the universal cover \(\tilde{U}\) by deck transformations and on \(\C^2\) by the  lift. The connection \(\nabla'\) is defined so that it pulls back to the trivial connection on \(\tilde{U} \times \C^2\).\footnote{See also \cite[Remark 2.1]{LorayPerez}. More intrinsically, a projective structure on a Riemann surface \(X\) induces a holomorphic connection on \(1\)-jets of \((TX)^{1/2}\).}
\end{proof}

\begin{lemma}
	There is a unique rank \(2\) holomorphic vector bundle \(E'\) over \(\CP^1\) such that its restriction to \(U\) is equal to \(E'_U\) and
	\(\nabla'\) extends as a logarithmic connection on \(E'\) with residues at the points \(x_i\) conjugate to
	\begin{equation*}
	A_i = \begin{pmatrix}
	a_{i1} & 0 \\
	0 & a_{i2}
	\end{pmatrix} 
	\end{equation*}
	where \(a_{i1}, a_{i2}\) are given by Equation \eqref{eq:weights}.
\end{lemma}

\begin{proof}
	This is standard, we give a sketch idea, for details see
	Proposition 5.4 in \cite[p. 94]{Deligne}.
	Since \(\exp(2\pi i A_i)\) is conjugate to the holonomy of \(\nabla'\) about \(x_i\) (see Equation \eqref{eq:holloop}), we can find holomorphic sections \(s_1, s_2\) of \(E_U'\)
	on a small punctured disc around \(x_i\), uniquely determined up to multiplication by constant factors, such that in the associated trivialization
	\begin{equation}\label{eq:Ai}
		\nabla' = d - \frac{A_i}{t} dt .
	\end{equation}
	We extend \((E'_U, \nabla')\) over \(x_i\) using the trivialising frame \(s_1, s_2\). Clearly, the extended connection has logarithmic singularities of the required type.
	
	Suppose \((E'', \nabla'')\) is another such extension. 
	Close to \(x_i\) the two connections \(\nabla'\) and \(\nabla''\) are equivalent to \eqref{eq:Ai} and the identity isomorphism outside \(x_i\) (being a flat section of the endomorphism bundle with the induced connection) is represented by a diagonal matrix with constant entries. The constants clearly extend over the origin, showing that the two bundles \(E''\) and \(E'\) are equal.
\end{proof}

\begin{lemma}\label{lem:degree}
	 \(\deg E' =n\).
\end{lemma}

\begin{proof}
	The residue theorem for meromorphic connections (\cite[Corollary 17.35]{IY}) asserts that
	\[\deg E' = \sum_i (a_{i1} + a_{i2}) . \]
	In our case, \(a_{i1}+a_{i2} = 1\) for each \(i\).
\end{proof}

Same as before, the flat local sections of \(E'\) define a holomorphic foliation \(\mF'\) on the projectivization \(\mathbb{P}(E'|_U)\) transversal to the fibres.
At each point \(x_i\) the connection \(\nabla'\) is locally equivalent to the model given by Equation \eqref{eq:modelconection}.
The proof of Lemma \ref{lem:extF} shows that \(\mF'\) extends
to a singular foliation on \(\mathbb{P}(E')\) and the extension is tangent to the fibre over each \(x_i\) with two linearisable singularities locally equivalent to the ones of \(\mF\). 

\begin{lemma}\label{lem:sigma'}
	The developing map of \(g'\) defines a holomorphic section \(\sigma'\) of \(\mathbb{P}(E')\) that is everywhere transversal to \(\mF'\).
\end{lemma}

\begin{proof}
	Let \(\tilde{U}\) be the universal cover of the punctured sphere \(U\) with the pull-back complex structure. Taking a branch of the developing map of \(g'|_U\) we obtain a holomorphic function \(\tilde{\sigma}: \tilde{U} \to \CP^1\) with nowhere zero derivative. The function
	\(\tilde{\sigma}\) is equivariant with respect to the actions of \(\pi_1(U)\) by deck-transformations and holonomy of \(\nabla'\) respectively, so it defines a section \(\sigma'\) of \(\mathbb{P}(E')|_U = \tilde{U} \times_{\pi_1(U)} \CP^1\) with the required properties.
	
	At a singularity \(x_i\)
	we have a centred complex coordinate \(z\) such that
	\(g'\) is the pull-back of the Fubini-Study metric by \(z \mapsto z^{\alpha_i}\). We are reduced to the model case of the trivial rank \(1\) bundle with connection
	\[d + \frac{\alpha_i}{z} dz . \]
	The (multivalued) developing map \(z \mapsto z^{\alpha_i}\) from a neighbourhood of the origin on the base to a neighbourhood of the origin at the fibre over \(1\) corresponds to the section of the trivial bundle given by \(f(z) \equiv 1\).
	From this local representation of the section we conclude that \(\sigma'\) extends holomorphically over \(x_i\) avoiding the two singularities of \(\mF'\).
\end{proof}

\begin{lemma}\label{lem:selfin}
	The section \(\sigma'\) of \(\mathbb{P}(E')\) has self-intersection equal to
	\begin{equation}\label{eq:square}
		\sigma'^2 = -(n-2) .
	\end{equation}
\end{lemma}

\begin{proof}
	Lemma \ref{lem:sigma'} implies that \(\mbox{Tan}(\mF', \sigma') = 0\) and the proof of Lemma \ref{lem:transverse} shows that
	\[\mbox{Tan}(\mF', \sigma') = \sigma'^2 + (n-2) . \]
	Equation \eqref{eq:square} follows.
\end{proof}

\begin{lemma}\label{lem:iso}
	The vector bundle \(E'\) is isomorphic to \(E=\mO(1)\oplus \mO(n-1)\).
\end{lemma}

\begin{proof}
	By the Birkhoff-Grothendieck theorem, \(E'\) is isomorphic to a direct sum \(\mO(a) \oplus \mO(b)\) with \(a \leq b\). Lemma \ref{lem:degree} implies that \(a+b=n\). Lemma \ref{lem:selfin} implies that \(\mathbb{P}(E') = \Sigma_{n-2}\) and so \(b-a=n-2\). We conclude that \(a=1\) and \(b=n-1\).
\end{proof}

The unitary logarithmic connection \(\nabla'\) endows \(E'\) with a semi-stable parabolic structure \(E'_{*}\) such that \(\nabla'\) is adapted to \(E'_{*}\). The weights are given by Equation \eqref{eq:weights} and \(\mathbf{F}' = \{F'_1, \ldots F'_n\} \) are the \(a_{i1}\)-eigenspaces of the residues of \(\nabla'\). 

\begin{lemma}\label{lem:isopar}
	There is an isomorphism (unique up to scaling) of \(E'\) with \(E\) that takes the flag \(\mathbf{F}'\) to the flag \(\mathbf{F}\) given by Lemma \ref{lem:flag}.
\end{lemma}

\begin{proof}
	Let us identify \(E'\) with \(E\) by an isomorphism, as guaranteed by Lemma \ref{lem:iso}. The line bundle given by \(\sigma'\) is equal to \(\mO(n-1)\). Since \(\sigma'\) does not go through the singularities of \(\mF'\), we conclude that \(\mathbf{F}' \cap \mO(n-1)\) is zero at all fibres \(E_{x_i}\). On the other hand, we claim that there is no degree one line bundle \(L' \subset E'\) such that \(\mathbf{F}' \subset L'\). Indeed, if that were the case then
	\[\pardeg L_{*}' = 1 - \frac{1}{2} \sum_i (1-\alpha_i) >0 \]
	by Gauss-Bonnet \eqref{eq:GB}, contradicting the semi-stability of \(E'_{*}\). We conclude that \(\mathbf{F}'\) satisfies properties \((i)\) and \((ii)\) of Lemma \ref{lem:flag} and the statement follows.
\end{proof}

\begin{lemma}
The spherical metric in Theorem \ref{thm} is unique, namely	\(g'=g\).
\end{lemma}

\begin{proof}
	Identify the parabolic bundles \(E'_{*}\) and \(E_{*}\) by an isomorphism as provided by Lemma \ref{lem:isopar}.
	Since the parabolic bundle is stable, the unitary connection given by the Mehta-Seshadri theorem is unique. We conclude that \(\nabla'=\nabla\) and, by construction, the two metrics \(g'\) and \(g\) are equal.
\end{proof}

\bibliographystyle{alpha}
\bibliography{BIB}

\end{document}